\documentclass[a4paper,11pt]{article}
\usepackage{amsmath,amsthm,amsfonts,amssymb,color,enumerate,xfrac,tikz,comment,pgfplots}
\usepackage[colorlinks,citecolor=blue,urlcolor=blue]{hyperref}
\usetikzlibrary{shapes.geometric,positioning,matrix}
\excludecomment{codes}
\usepackage[text={6.5in,9.0in},centering]{geometry}

\usepackage{graphicx,url,subfig}
\RequirePackage[OT1]{fontenc}
\usepackage{datetime}
\usepackage[normalem]{ulem} 
\usepackage{soul} 
\makeatother
\numberwithin{equation}{section}
\allowdisplaybreaks

\usepackage{amssymb}
\usepackage{pifont}
\theoremstyle{plain}
\newtheorem{theorem}{Theorem}[section]
\newtheorem{corollary}[theorem]{Corollary}
\newtheorem{lemma}[theorem]{Lemma}

\theoremstyle{definition}

\newtheorem{remark}[theorem]{Remark}

\newtheorem{example}[theorem]{Example}

\newcommand{\E}{\mathbb{E}}

\newcommand{\ud}{\ensuremath{\mathrm{d} }}

\newcommand{\Norm}[1]{\left|\left|  #1   \right|\right|}

\newcommand{\R}{\mathbb{R}}

\DeclareMathOperator{\Lip}{\mathit{L}}

\usepackage[strict=true,style=english]{csquotes}
\usepackage[backend=biber,
            style=alphabetic,
            natbib=true,
            maxbibnames=99,
            sorting=nyt,
            abbreviate=true]{biblatex}
\AtEveryBibitem{
 \clearfield{doi}
 \clearfield{url}
 \clearfield{issn}
 \clearlist{location}
 \clearfield{month}
 \clearfield{series}

 \ifentrytype{book}{}{
  \clearlist{publisher}
  \clearname{editor}
 }
}
\title{Superlinear stochastic heat equation on $\R^d$}
\author{
  Le Chen\footnote{Department of Mathematics and Statistics, Auburn University, Auburn, Alabama, USA. Email: \url{le.chen@auburn.edu}.}
  \and
  Jingyu Huang\footnote{School of Mathematics, University of Birmingham, Birmingham, UK. Email: \url{j.huang.4@bham.ac.uk}.}
}
\begin{document}
\maketitle

\begin{abstract}

  In this paper, we study the {\it stochastic heat equation} (SHE) on $\R^d$ subject to a centered
  Gaussian noise that is white in time and colored in space. We establish the existence and
  uniqueness of the random field solution in the presence of locally Lipschitz drift and diffusion
  coefficients, which can have certain superlinear growth. This is a nontrivial extension of the
  recent work by Dalang, Khoshnevisan and Zhang \cite{dalang.khoshnevisan.ea:19:global}, where the
  one-dimensional SHE on $[0,1]$ subject to space-time white noise has been studied.

  \bigskip
	\noindent{\it Keywords.} Global solution; Stochastic heat equation; Reaction-diffusion; Dalang's
  condition; superlinear growth.\\

	\noindent{\it \noindent }

\end{abstract}

\setlength{\parindent}{1.5em}

\section{Introduction}

In this paper, we study the following stochastic heat equation (SHE) on $\R^d$ with a drift term:
\begin{equation} \label{E:SHE}
  \begin{cases}
    \dfrac{\partial u(t,x)}{\partial t} = \dfrac{1}{2}\Delta u(t,x) + b\left(u(t,x)\right) +
    \sigma\left(u(t,x)\right)\dot{W}(t,x)\,, & t>0, \: x\in\R^d, \\[1em]
    u(0,\cdot) = u_0(\cdot),
  \end{cases}
\end{equation}
with both $b$ and $\sigma$ being locally Lipschitz continuous. The noise $\dot{W}$ is a centered
Gaussian noise which is white in time and colored in space with the following covariance structure
\begin{equation} \label{E:Noise}
  \E\left[\dot{W}(s,y)\dot{W}(t,x)\right] = \delta(t-s)f(x-y)\,.
\end{equation}
We assume that the correlation function $f$ in \eqref{E:Noise} satisfies the improved {\it Dalang's
condition}:
\begin{align} \label{E:Dalang+}
  \Upsilon_\alpha:= (2\pi)^{-d}\int_{\R^d} \frac{\hat{f}(\xi)\ud\xi}{(1+|\xi|^2)^{1-\alpha}} < \infty\,,
  \quad \text{for some $0 < \alpha < 1$,}
\end{align}
where $\hat{f}(\xi)$ is the Fourier transform of $f$, namely, $\hat{f}=\mathcal{F}f(\xi) =
\int_{\R^d} f(x) e^{-i x\cdot \xi} \ud x$. Recall that condition \eqref{E:Dalang+} with $\alpha=0$
refers to {\it Dalang's condition} \cite{dalang:99:extending}:
\begin{align}\label{E:Dalang}
  \Upsilon(\beta):=(2\pi)^{-d}\int_{\R^d} \frac{\hat{f}(\ud \xi)}{\beta+|\xi|^2}<+\infty \quad \text{for some and hence for all $\beta>0$.}
\end{align}
The case when $f=\delta_0$ refers to the space-time white noise. The solution to \eqref{E:SHE} is
understood in the mild formulation:
\begin{equation} \label{E:Mild}
  \begin{aligned}
    u(t,x) = (p_t*u_0)(x) & +\int_0^t \int_{\R^d} p_{t-s}(x-y)b(u(s,y))\ud y \ud s \\
                          & +\int_0^t \int_{\R^d} p_{t-s}(x-y)\sigma(u(s,y))W(\ud s, \ud y)\,,
  \end{aligned}
\end{equation}
where the stochastic integral is the {\it Walsh integral}
\cite{walsh:86:introduction,dalang.khoshnevisan.ea:09:minicourse},
$p_t(x)=\left(2\pi t\right)^{-d/2}\exp\left(-|x|^2/2t\right)$ is the heat kernel, and ``$*$" denotes
the convolution in the spatial variable. \bigskip

Motivated by the work of Fernandez Bonder and Groisman
\cite{fernandez-bonder.groisman:09:time-space}, Dalang, Khoshnevisan and Zhang
\cite{dalang.khoshnevisan.ea:19:global} established the {\it global solution with superlinear and
locally Lipschitz coefficients} for the one-dimensional SHE on $[0,1]$ subject to the space-time
white noise. In particular, they assumed that
\begin{align} \label{E:Davar}
  |b(z)| = O\left(|z|\log|z|\right) \quad \text{and} \quad
  |\sigma(z)| = o\left(|z|\left(\log|z|\right)^{1/4}\right).
\end{align}
Foondun and Nualart \cite{foondun.nualart:21:osgood} studied SHE with an additive noise, i.e.,
$\sigma(\cdot) \equiv \text{const.}$, and showed that the solution to \eqref{E:SHE} blows up in
finite time if and only if $b$ satisfies the {\it Osgood condition}:
\begin{align} \label{E:Osgood}
  \int_c^\infty \frac{1}{b(u)}\ud u<\infty \quad \text{for some $c>0$}.
\end{align}
Salins \cite{salins:21:global} studied this problem for SHE on a compact domain in $\R^d$ under the
following Osgood-type conditions, which are weaker than \eqref{E:Davar}: There exists a positive and
increasing function $h:[0,\infty)\to[0,\infty)$ that satisfies
\begin{align*}
  \int_c^\infty \frac{1}{h(u)} \ud u=\infty \quad \text{for all $c>0$}
\end{align*}
such that for some $\gamma\in (0,1/2)$ (which depends on the noise),
\begin{align} \label{E:Salins}
  |b(z)     | \le h\left( | z | \right) \quad  \text{for all $z \in\R$}, \quad \text{and} \quad
  |\sigma(z)| \le | z | ^{1-\gamma}\left( h\left(| z | \right)\right)^{\gamma} \quad  \text{for
  all $z >1$}.
\end{align}
Extending the above results to the SHE on the whole space $\R^d$ with both superlinear drift and
diffusion coefficients is a challenging problem due to the non-compactness of the spatial domain.
Indeed, for the wave equation on $\R^d$ ($d=1,2,3$), the compact support of the corresponding
fundamental solution can help circumvent this difficulty; see Millet and Sanz-Sol\'e
\cite{millet.sanz-sole:21:global}. The aim of this present paper is to carry out such extension by
proving the following theorem:

\begin{theorem} \label{T:Main}
  Assume the improved Dalang's condition \eqref{E:Dalang+} is satisfied for some $\alpha\in (0,1)$.
  Let $u(t,x)$ be the solution to \eqref{E:SHE} starting from $u_0\in L^\infty(\R^d)\cap L^p(\R^d)$
  for some $p> (d+2)/\alpha$. Suppose that $b$ and $\sigma$ are locally Lipschitz functions such
  that $b(0)=\sigma(0)=0$.
  \begin{enumerate}[(a)]
    \item (Global solution) If
    \begin{align} \label{E:subCrtcl}
      \max\left(\frac{|b(z)|}{\log|z|},\: \frac{\left|\sigma(z)\right|}{\left(\log |z|\right)^{\alpha/2}} \right)
      = o(|z|) \quad \text{as $|z|\to \infty$,}
    \end{align}
    then for any $T>0$, there is a unique solution $u(t,x)$ to \eqref{E:SHE} for all $(t,x)\in
    (0,T]\times\R^d$.
  \item (Local solution) If
    \begin{align} \label{E:Crtcl}
      \max\left(\frac{|b(z)|}{\log|z|},\: \frac{\left|\sigma(z)\right|}{\left(\log |z|\right)^{\alpha/2}} \right)
      = O(|z|) \quad \text{as $|z|\to \infty$,}
    \end{align}
    then for some deterministic time $T>0$, there exists a unique solution solution $u(t,x)$ to
    \eqref{E:SHE} for all $(t,x)\in (0,T]\times\R^d$.
  \item In either case (a) or (b), the solution $u(t,x)$ is H\"older continuous: $u\in C^{\alpha/2
    -,\, \alpha - }\left((0,T]\times\R^d\right)$ a.s., where $C^{\alpha_1 -, \,
    \alpha_2-}\left(D\right)$ denotes the H\"older continuous function on the space-time domain $D$
    with exponents $\alpha_1-\epsilon$ and $\alpha_2-\epsilon$ in time and space, respectively, for
    any small $\epsilon>0$.
  \end{enumerate}
\end{theorem}

Theorem \ref{T:Main} is proved in Section \ref{S:Main}.

\begin{remark}[Critical vs sub-critical cases] \label{R:Critical}
  We call the case under conditions in \eqref{E:Crtcl} the {\it critical case} and the one in
  \eqref{E:subCrtcl} the {\it sub-critical case}. Dalang {\it et al}
  \cite{dalang.khoshnevisan.ea:19:global} established the global solution for the critical case
  using the semigroup property of the heat equation. In this paper, we cannot restart our SHE
  \eqref{E:SHE} to pass the local solution to global solution due to the fact that it is not clear
  whether at time $T$, $u(T,\cdot)$ as the initial condition for the next step is again an element
  in $L^\infty(\R^d)\cap L^p(\R^d)$ a.s. This issue does not present for a continuous random field
  on a compact spatial domain, which is the case in \cite{dalang.khoshnevisan.ea:19:global} and
  \cite{salins:21:global}.
\end{remark}

\begin{remark}[Regularity of the initial conditions] \label{R:Init}
  In \cite{dalang.khoshnevisan.ea:19:global}, the initial condition $u_0$ is assumed to be a
  H\"older continuous function on $[0,1]$. In contrast, in this paper, we only assume that $ u_0\in
  L^\infty(\R^d)\cap L^p(\R^d)$ for some large $p$. This is one example of the smoothing effect of
  the heat kernel in the stochastic partial differential equation context. This improvement from a
  H\"older continuous function to a measurable function is due to the factorization representation
  of the solution (see \eqref{E:Factorize} below). Similar arguments using this factorization have
  also been carried out by Salins \cite{salins:21:global}.
\end{remark}

\begin{example}[Examples of $b$ and $\sigma$ in Theorem \ref{T:Main}] \label{Eg:Egs}
  (1) The function $g(x) = x\sin(x)$ for $x\in\R$ is locally Lipschitz, but not globally Lipschitz,
  continuous with linear growth and $g(0)=0$. Hence Theorem \ref{T:Main} holds when either $b$ or
  $\sigma$ takes the form of $g$. (2) For the function $g_{a,b}(x):=|x|^b\log^a(1+|x|)$ for $x, a,
  b\in\R$, it is easy to see that the conditions $a+b>0$ and $a+b\ge 1$ imply that $g_{a,b}(0)=0$
  and $g_{a,b}$ is locally Lipschitz continuous, respectively. The growth condition of either
  \eqref{E:subCrtcl} or \eqref{E:Crtcl} makes the further restriction on the suitable choices of
  $(a,b)$.
\end{example}

%
%
%

The extension given in Theorem \ref{T:Main} from a compact spatial domain to the entire space $\R^d$
critically relies on the sharp moment formulas obtained in Theorem \ref{T:Mom} below. These moment
formulas, as extensions of those in
\cite{chen.dalang:15:moments,chen.kim:19:nonlinear,chen.huang:19:comparison} to allow a Lipschitz
drift term, constitute the second and independent contribution of the paper. Indeed, when the drift
term is linear, i.e., $b(u) = \lambda u$, then one can work with the following heat kernel $G_d(t,x)
= p_t(x) e^{\lambda t}$. However, when $b$ is a Lipschitz nonlinear function, the situation is much
more trickier, especially if one wants to allow {\it rough initial conditions}
\cite{chen.dalang:15:moments,chen.kim:19:nonlinear,chen.huang:19:comparison}, namely, $u_0$ being a
signed Borel measure such that
\begin{align} \label{E:InitData}
  \int_{\R^d} e^{-a|x|^2} |u_0|(\ud x)<\infty \quad \text{for all $a>0$},
\end{align}
where $|u_0| = u_{0,+} + u_{0,-}$ and $u_0=u_{0,+} - u_{0,-}$ is the Jordan decomposition of the
signed measure $u_0$. The existence and uniqueness of the solution $u$ is proved in
\cite{huang:17:on} (the proof still works for signed Borel measure). We will prove the following
theorem:

\begin{theorem}[Moment formulas with a Lipschitz drift term] \label{T:Mom}
  Let $u(t,x)$ be the solution to \eqref{E:SHE} and suppose that $b$ and $\sigma$ are globally
  Lipschitz continuous functions and the correlation function $f$ satisfies the improved Dalang's
  condition \eqref{E:Dalang+} for some $\alpha\in(0,1)$. Then we have the following:
  \begin{enumerate}[(a)]
    \item If $u_0\in L^\infty(\R^d)$, then for all $p\ge \max\left(2,\: 2^{-6} L_b^{-2}
      \Upsilon_\alpha^{-1}\right)$, $t>0$ and $x\in\R^d$, it holds that
      \begin{equation} \label{E:Mom}
        \Norm{u(t,x)}_p \leq \left( \frac{\tau}{2}+ 2\Norm{u_0}_{L^\infty}\right)\exp\left(C t \max\left(p^{1/\alpha}\Lip_\sigma^{2/\alpha},\Lip_b\right)\right) \,,
      \end{equation}
      where $\Norm{\cdot}_p$ and $\Norm{\cdot}_{L^\infty}$ denote the $L^p\left(\Omega\right)$-norm
      and $L^\infty(\R^d)$-norm, respectively,
      \begin{align} \label{E:tau}
        \tau:= \frac{ |b(0)|}{L_b}\vee \frac{|\sigma(0)|}{L_{\sigma}}\,,
      \end{align}
      $C=\max\left(4, 2^{6/\alpha-1}\Upsilon_\alpha^{1/\alpha}\right)$, and
      \begin{align} \label{E:Lb}
        L_b      := \sup_{z\in\R} \frac{|b(z) - b(0)|}{|z|} \quad \text{and} \quad
        L_\sigma := \sup_{z\in\R} \frac{|\sigma(z) - \sigma(0)|}{|z|}.
      \end{align}
    \item If $u_0$ is a rough initial condition (see \eqref{E:InitData}), then for all $t>0$,
      $x\in\R^d$ and $p\ge 2$,
      \begin{align} \label{E:MomRough}
        \Norm{u(t,x)}_p \le
        \sqrt{3} \left[\tau + J_+(t,x)\right] \exp\left(C t\max\left(p^{1/\alpha}\Lip_\sigma^{2/\alpha},\Lip_b\right)\right),
      \end{align}
      where $J_+(t,x) := (p_t * |u_0|)(x)$ and the constant $C$ does not depend on
      $(t,x,p,L_b,L_\sigma)$.
    \item If $u_0\in L^{\infty}(\R^d)\cap L^p(\R^d)$ for some $p\ge (2+d)/\alpha$ and if $\sigma(0)
      = b(0) = 0$, then for all $t>0$,
      \begin{align} \label{E:MomBdd}
        \hspace{-1em}
        \Norm{\sup_{(s,x)\in [0,t]\times\R^d} u(s,x)}_p
        & \le \Norm{u_0}_{L^\infty} + C \Norm{u_0}_{L^p} \left(L_b+L_\sigma\right) \exp\left(Ct\max\left(p^{1/\alpha}L_{\sigma}^{2/\alpha},\: L_b\right)\right),
      \end{align}
      where $\Norm{\cdot}_{L^p}$ denotes the $L^p(\R^d)$-norm and the constant $C$ does not depend
      on $(t,x,p,L_b,L_\sigma)$.
  \end{enumerate}
\end{theorem}

\begin{remark}
  Part (a) of Theorem \ref{T:Mom} can be derived from part (b) by noticing that $J_+(t,x)\le
  \Norm{u_0}_{L^\infty}$. However, we still keep part (a) due to the simplicity of its proof.
\end{remark}

Using the moment bounds in \eqref{E:MomRough}, one can extend the H\"older regularity from the SHE
without drift (see \cite{sanz-sole.sarra:02:holder} for the bounded initial condition case and
\cite{chen.huang:19:comparison} for the rough initial condition case) to the one with a Lipschitz
drift.

\begin{corollary}[H\"older regularity] \label{C:Holder}
  Let $u(t,x)$ be the solution to \eqref{E:SHE} starting from a rough initial condition (see
  \eqref{E:InitData}) and suppose that $b$ and $\sigma$ are globally Lipschitz continuous functions.
  If the correlation function $f$ satisfies the improved Dalang's condition \eqref{E:Dalang+} for
  some $\alpha\in(0,1)$. Then $u\in C^{\alpha /2-, \, \alpha-}\left((0,\infty)\times\R^d\right)$
  a.s.
\end{corollary}

Parts (a), (b), and (c) of Theorem \ref{T:Mom} are proved in Sections \ref{SS:bddInit}, \ref{SS:MomRough}, and
\ref{SS:MomBdd}, respectively. Corollary \ref{C:Holder} is proved in Section \ref{SS:Holder}. \bigskip

Finally, we list a few open questions for future exploration: (1) Theorem \ref{T:Main} cannot handle
either the constant one initial condition or the Dirac delta initial condition. It is interesting to
investigate if either global or local solution exists for these two special initial conditions. (2)
Can one improve Theorem \ref{T:Main} by relaxing the growth conditions in \eqref{E:subCrtcl} and
\eqref{E:Crtcl} to the Osgood-type conditions in \eqref{E:Salins} as in \cite{salins:21:global}?
\medskip

In the rest of the paper, we prove Theorems \ref{T:Mom} and \ref{T:Main} in Sections \ref{S:Moment}
and \ref{S:Main}, respectively.

\section{Moment bounds with a Lipschitz drift term} \label{S:Moment}

\subsection{The bounded initial data case -- Proof of part (a) of Theorem \ref{T:Mom}} \label{SS:bddInit}

\begin{proof}[Proof of Theorem \ref{T:Mom} (a):]
  By Minkowski's inequality,
  \begin{align*}
    \Norm{u(t,x)}_p \leq
    & (p_t*u_0)(x) + \int_0^t \int_{\R^d} p_{t-s}(x-y)\left(|b(0)|+ L_b \Norm{u(s,y)}_p\right) \ud y \ud s                        \\
    & + z_{p}\bigg(\int_0^t \int_{\R^d} \int_{\R^d} p_{t-s}(x-y) p_{t-s}(x-y')\left(|\sigma(0)|+ L_{\sigma}\Norm{u(s,y)}_p\right) \\
    & \qquad \times\left(|\sigma(0)|+ L_{\sigma}\Norm{u(s,y')}_p\right) f(y-y') \ud y \ud y' \ud s\bigg)^{1/2}\,,
  \end{align*}
  where $z_p$ is the constant coming from the Burkholder-Davis-Gundy inequality and $z_p \sim 2
  \sqrt{p}$ as $p\to \infty$; see \cite[Theorem 1.4]{conus.khoshnevisan:12:on} and references
  therein. For $\beta>0$, consider the following norm
  \begin{align*}
    \mathcal{N}_{\beta}(u) := \sup_{(t,x)\in(0,\infty)\times\R^d} e^{-\beta t} \Norm{u(t,x)}_p\,.
  \end{align*}
  Then we see that
  \begin{align*}
         & e^{-\beta t}\Norm{u(t,x)}_p                                                                                                                                                             \\
    \leq & \|u_0\|_{L^{\infty}}  + \int_0^t \int_{\R^d} e^{-\beta(t-s)}p_{t-s}(x-y)\left(|b(0)|+ L_b \left(\sup_{(s,y)\in(0,\infty)\times\R^d} e^{-\beta s} \|u(s,y)\|_p\right)\right) \ud y \ud s \\
         & + z_{p}\bigg(\int_0^t \int_{\R^d} \int_{\R^d} e^{-2\beta (t-s)}p_{t-s}(x-y) p_{t-s}(x-y')                                                                                               \\
         & \qquad \times \left(|\sigma(0)|+ L_{\sigma}\sup_{(s,y)\in(0,\infty)\times\R^d}e^{-\beta s}\Norm{u(s,y)}_p\right)^2 f(y-y') \ud y \ud y' \ud s\bigg)^{1/2}\,.
  \end{align*}
  Hence,
  \begin{align*}
    \mathcal{N}_{\beta}(u) \leq
      & \Norm{u_0}_{L^{\infty}} + \frac{1}{\beta} \left(|b(0)| + L_b \: \mathcal{N}_{\beta}(u)\right) \\
      & + z_p \left( \left(2\pi\right)^{-d} \int_0^\infty \int_{\R^d} e^{-2\beta s} e^{-s |\xi|^2} \hat{f}(\xi) \ud\xi \ud s\right)^{1/2} \bigg(|\sigma(0)|+ L_{\sigma}\: \mathcal{N}_{\beta}(u)\bigg).
  \end{align*}
  By the improved Dalang's condition \eqref{E:Dalang+} and by assuming that $\beta>1/2$, we see that
  \begin{align*}
    \left(2\pi\right)^{-d} \int_0^\infty \int_{\R^d} e^{-2\beta s} e^{-s |\xi|^2} \hat{f}(\xi) \ud\xi \ud s
    & =\left(2\pi\right)^{-d} \int_{\R^d} \frac{\widehat{f}\left(\ud\xi\right)}{\left(2\beta+|\xi|^2\right)^{1-\alpha}\left(2\beta+|\xi|^2\right)^{\alpha}}
      \le (2\beta)^{-\alpha}\Upsilon_\alpha.
  \end{align*}
  Therefore,
  \begin{align*}
    \mathcal{N}_{\beta}(u)
    \le & \Norm{u_0}_{L^{\infty}}
        + \frac{L_b}{\beta} \left(\frac{|b(0)|}{L_b} + \mathcal{N}_{\beta}(u)\right)
        + z_p\left(2\beta\right)^{-\alpha/2} \Upsilon_\alpha^{1/2}L_\sigma\left(\frac{|\sigma(0)|}{L_\sigma}+\mathcal{N}_{\beta}(u)\right)\,.
  \end{align*}

  Now by choosing $\beta$ large enough, namely,
  \begin{align*}
    \beta>\frac{1}{2},\quad
    \frac{L_b}{\beta} \leq \frac{1}{4}, \quad
    z_p \left(2\beta\right)^{-\alpha/2} \Upsilon_\alpha^{1/2} L_{\sigma} \leq \frac{1}{4}\,
    \quad \Longleftrightarrow \quad
    \beta>\max\left(4L_b,\: \frac{1}{2},\:  \frac{1}{2}\left(16z_p^2 L_{\sigma}^2 \Upsilon_\alpha\right)^{1/\alpha}\right),
  \end{align*}
  we form a contraction map, which can be easily solved:
  \begin{align*}
    \mathcal{N}_{\beta}(u) \leq 2 \Norm{u_0}_{L^\infty} + \frac{|b_0|}{2L_b} \vee \frac{|\sigma(0)|}{2 L_{\sigma}}\,.
  \end{align*}
  Notice that $z_p\le 2 \sqrt{p}$, we have that
  \begin{align*}
    \frac{1}{2}< \frac{1}{2}\left(16z_p^2 L_{\sigma}^2 \Upsilon_\alpha\right)^{1/\alpha}
    \quad \Longleftrightarrow \quad
    1/p < 64 L_b^2 \Upsilon_\alpha.
  \end{align*}
  Therefore, for all $t>0$, when $1/p < \min\left(64 L_b^2 \Upsilon_\alpha,1/2\right)$, we can take
  \begin{align*}
    \max\left(4L_b,\: \frac{1}{2}\left(16 (2\sqrt{p})^2 L_{\sigma}^2 \Upsilon_\alpha\right)^{1/\alpha}\right)
    \le \max\left(4, 2^{6/\alpha-1}\Upsilon_\alpha^{1/\alpha}\right)\max\left(L_b,\:p^{1/\alpha} L_\sigma^{2/\alpha} \right)
    =:\beta
  \end{align*}
  to have that
  \begin{align*}
    \Norm{u(t,x)}_p \le \left(2 \Norm{u_0}_{L^\infty} + \frac{|b_0|}{2L_b} \vee \frac{|\sigma(0)|}{2 L_{\sigma}}\right)
    \exp\left(\beta t\right), \quad \text{for all $t>0$}.
  \end{align*}
  This proves part (a) of Theorem \ref{T:Mom}.
\end{proof}

\subsection{A Gronwall-type lemma} \label{SS:Grownwal}

Let us introduce some functions. For $a,b\ge 0$, denote
\begin{align}\label{E:k}
  k_{a,b}(t):=\int_{\R^d}\left(af(z)+bt\right)G(t,z)\ud z = a k_{1,0}(t) + bt.
\end{align}
By the Fourier transform, this function can be written in the following form
\begin{align}\label{E:k2}
  k_{1,0}(t):=(2\pi)^{-d} \int_{\R^d}\hat{f}(\ud\xi)\exp\left(-\frac{t|\xi|^2}{2}\right).
\end{align}
Define $h_0^{a,b}(t):=1$ and for $n\ge 1$,
\begin{align}\label{E:hn}
  h_n^{a,b}(t)= \int_0^t \ud s \: h_{n-1}^{a,b}(s) k_{a,b}(t-s).
\end{align}
Let
\begin{align}\label{E:H}
  H_{a,b}(t;\gamma):= \sum_{n=0}^\infty \gamma^n h_n^{a,b}(t),\qquad\text{for all $\gamma\ge 0$.}
\end{align}
When we have $a=1$ and $b=0$, we will use $k(t)$,  $h_n(t)$ and $H(t;\gamma)$ to denote
$k_{1,0}(t)$, $h_n^{1,0}(t)$ and $H_{1,0}(t;\gamma)$, respectively. Note that this convention makes
our notation in case of $a=1$ and  $b=0$ consistent with those in \cite{chen.huang:19:comparison},
\cite{chen.kim:19:nonlinear} or \cite{balan.chen:18:parabolic}. The following lemma generalizes
Lemma 2.5 in \cite{chen.kim:19:nonlinear} or Lemma 3.8 in \cite{balan.chen:18:parabolic} from the
case $a=1$ and $b=0$ to the case with general parameters $a$ and $b$.

\begin{lemma} \label{L:EstHt}
  Suppose that the correlation function $f$ satisfies Dalang's condition \eqref{E:Dalang}. Then for
  all $a\ge 0$,  $b\ge 0$, and $\gamma\ge 0$, it holds that
  \begin{align} \label{E:Var}
    \limsup_{t\rightarrow\infty} \frac{1}{t}\log H_{a,b}(t;\gamma)
    \le  \inf\left\{\beta>0: \:a \Upsilon\left(2\beta\right) + \frac{b}{2\beta^2} < \frac{1}{2\gamma}\right\},
  \end{align}
  where $\Upsilon(\beta)$ is defined in \eqref{E:Dalang}.
\end{lemma}
\begin{proof}
  Here we follow the arguments in the proof of Lemma 3.8 of \cite{balan.chen:18:parabolic}. In
  particular,
  \begin{align*}
    \limsup_{t \to \infty}\frac{1}{t}\log H_{a,b}(t;\gamma) \leq
    \inf \left\{\beta>0; \int_0^{\infty}e^{-\beta t}H_{a,b}(t;\gamma)dt<\infty\right\}.
  \end{align*}
  Notice that
  \begin{align*}
    \int_0^{\infty}e^{-\beta t}H_{a,b}(t;\gamma)\ud t
    = & \sum_{n \geq 0}\gamma^n \int_0^{\infty}e^{-\beta t}h_n^{a,b}(t)\ud t                                                                                                  \\
    = & \sum_{n \geq 0}\gamma^n \left[\int_0^{\infty}e^{-\beta t}k_{a,b}(t)\ud t\right]^n \left[\int_0^{\infty}e^{-\beta t}h_0^{a,b}(t)\ud t\right]                           \\
    = & \frac{1}{\beta}\sum_{n \geq 0}\gamma^n \left[a \int_0^{\infty}e^{-\beta t}k(t)\ud t + \frac{b}{\beta^2} \right]^n                                                     \\
    = & \frac{1}{\beta}\sum_{n \geq 0}\gamma^n \left[a \left(2\pi\right)^{-d}\int_{\R^d} \frac{\widehat{f}(\ud\xi)}{\beta + \frac{|\xi|^2}{2}}  + \frac{b}{\beta^2} \right]^n \\
    = & \frac{1}{\beta}\sum_{n \geq 0}\gamma^n \left[2 a \Upsilon(2\beta)+ \frac{b}{\beta^2} \right]^n,
  \end{align*}
  where in the fourth equality we have used \eqref{E:k2}. The lemma is proved by noticing that
  \begin{align*}
    \int_0^{\infty}e^{-\beta t}H_{a,b}(t;\gamma)\ud t <\infty \quad \Longleftrightarrow \quad
    2 a \Upsilon(2\beta)+ \frac{b}{\beta^2} < \frac{1}{\gamma}.
  \end{align*}
  One may check the proof of Lemma 3.8 of \cite{balan.chen:18:parabolic} for more details. This
  proves the lemma.
\end{proof}

\begin{corollary} \label{C:Growth}
  Suppose that the correlation function $f$ satisfies the improved Dalang's condition
  \eqref{E:Dalang+} for some $\alpha\in (0,1)$.  Then for all $a\ge 0$ and $b\ge 0$, when $\gamma>0$
  is large enough, it holds that
  \begin{align} \label{E:Var}
    \limsup_{t\rightarrow\infty} \frac{1}{t}\log H_{a,b}(t;\gamma)
    \le \max\left(2^{3/\alpha} \left(a C \gamma \right)^{1/\alpha},  \sqrt{2b\gamma}\right),
  \end{align}
  where the constant $C$ can be chosen to be
  \begin{align} \label{E:ConstC}
    C=\left(2\pi\right)^{-d} 2^{-\alpha}\max\left(\int_{|\xi|\le 1}\widehat{f}(\ud\xi), \int_{|\xi|>1}\frac{\widehat{f}(\ud\xi)}{|\xi|^{2(1-\alpha)}} \right).
  \end{align}
\end{corollary}
\begin{proof}
  Notice that for $\beta>0$,
  \begin{align*}
    \Upsilon(2\beta) & =(2\pi)^{-d}\int_{\R^d} \frac{1}{\left(2\beta+|\xi|^2\right)^{\alpha}} \frac{\hat{f}(\ud \xi)}{\left(2\beta+|\xi|^2\right)^{1-\alpha}} \\
                     & \le \frac{(2\pi)^{-d}}{(2\beta)^\alpha} \left(\int_{|\xi|\le 1} \frac{\hat{f}(\ud \xi)}{(2\beta)^{1-\alpha}} + \int_{|\xi|> 1} \frac{\hat{f}(\ud \xi)}{|\xi|^{2(1-\alpha)}}\right)
		      \le  C \left(\frac{1}{\beta}+\frac{1}{\beta^\alpha}\right),
  \end{align*}
  where the constant $C$ can be chosen as in \eqref{E:ConstC}. When $\gamma$ is large enough, we may
  assume that $\beta>1$. Hence, in light of \eqref{E:Var},
  \begin{align*}
    a \Upsilon\left(2\beta\right) + \frac{b}{2\beta^2} \le \frac{2aC}{\beta^\alpha} + \frac{b}{2\beta^2}.
  \end{align*}
  Therefore,
  \begin{align*}
    a \Upsilon\left(2\beta\right) + \frac{b}{2\beta^2} < \frac{1}{2\gamma}
    \quad & \Longleftarrow \quad \frac{2a C}{\beta^\alpha} + \frac{b}{2\beta^2} < \frac{1}{2\gamma}                                          \\
    \quad & \Longleftarrow \quad \frac{2a C}{\beta^\alpha} < \frac{1}{4\gamma} \quad \text{and} \quad \frac{b}{2\beta^2} < \frac{1}{4\gamma} \\
    \quad & \Longleftrightarrow \quad \beta> 2^{3/\alpha} \left(a C \gamma \right)^{1/\alpha} \quad \text{and} \quad \beta > \sqrt{2b\gamma}.
  \end{align*}
  This proves the corollary.
\end{proof}

\subsection{Moment bounds for rough initial data -- Proof of part (b) of Theorem \ref{T:Mom}} \label{SS:MomRough}

In this part, we extend the moment bounds obtained in \cite{chen.huang:19:comparison} to allow a
Lipschitz drift term.

\begin{proof}[Proof of Theorem \ref{T:Mom} (b)]
Taking the $p$-th norm on both sides of the mild form \eqref{E:Mild} with $p\ge 2$ and applying the
Minkowski inequality, we see that
\begin{align} \label{E:Normp}
  \Norm{u(t,x)}_p \le J_+(t,x) + \Lip_b \int_0^t\ud s\int_{\R^d} p_{t-s}(x-y) \left(\frac{|b(0)|}{L_{b}} + \Norm{u(s,y)}_p \right) \ud y + \Norm{I(t,x)}_p.
\end{align}
By the Burkholder-Davis-Gundy inequality (see also a similar argument in the step 1 of the proof of
Theorem 1.7 of \cite{chen.huang:19:comparison} on p. 1000), we see that
\begin{align*}
  \Norm{I(s,y)}_p^2 \le & 4p \Lip_{\sigma}^2 \int_0^s\iint_{\R^{2d}} p_{s-r}(y-z_1)p_{s-r}(y-z_2)f(z_1-z_2) \\
                        & \times
			\sqrt{2\left(\frac{\sigma(0)^2}{L_{\sigma}^2}+\Norm{u(r,z_1)}_p^2\right)}\sqrt{2\left(\frac{\sigma(0)^2}{L_{\sigma}^2}+\Norm{u(r,z_2)}_p^2\right)}\: \ud r \ud z_1 \ud z_2.
\end{align*}
Then by the sub-additivity of square root,
\begin{align} \label{E:MomentI}
  \begin{aligned}
  \Norm{I(t,x)}_p^2 \le & 8p \Lip_\sigma^2 \int_0^t\: \ud s\iint_{\R^{2d}} \ud y_1 \ud y_2\: p_{t-s}(x-y_1)p_{t-s}(x-y_2)f(y_1-y_2) \\
                        & \times \left(\frac{|\sigma(0)|}{L_{\sigma}}+\Norm{u(s,y_1)}_p\right)\left(\frac{|\sigma(0)|}{L_{\sigma}}+\Norm{u(s,y_2)}_p\right).
  \end{aligned}
\end{align}
By the Cauchy-Schwartz inequality applied to the $\ud t$ integral, the square of second term on the
right-hand side of \eqref{E:Normp} is bounded by
\begin{gather*}
  L_b^2\: t \int_0^t\ud s\left(\int_{\R^d} p_{t-s}(x-y) \left(\frac{|b(0)|}{L_b} + \Norm{u(s,y)}_p \right) \ud y \right)^2.
\end{gather*}
Hence, by raising both sides of \eqref{E:Normp} by a power two and recalling that the constant
$\tau$ is defined in \eqref{E:tau}, we obtain that
\begin{align*}
  \Norm{u(t,x)}_p^2 & \le 3 J_+^2(t,x) + 3 \Lip_b^2 t \int_0^t\ud s\left(\int_{\R^d} p_{t-s}(x-y) \left(\frac{|b(0)|}{L_b} + \Norm{u(s,y)}_p \right) \ud y \right)^2 \\
                    & \quad + 24 p\Lip_\sigma^2 \int_0^t \ud s\iint_{\R^{2d}} \ud y_1 \ud y_2\: p_{t-s}(x-y_1)p_{t-s}(x-y_2)f(y_1-y_2)                               \\
                    & \quad \times \left(\frac{|\sigma(0)|}{L_{\sigma}}+\Norm{u(s,y_1)}_p\right)\left(\frac{|\sigma(0)|}{L_{\sigma}}+\Norm{u(s,y_2)}_p\right)        \\
                    & \le 3 J_+^2(t,x) +  3 \int_0^t\ud s \iint_{\R^{2d}} \ud y_1 \ud y_2\: p_{t-s}(x-y_1)p_{t-s}(x-y_2)                                             \\
                    & \quad \times\left(8p\Lip_\sigma^2 f(y_1-y_2)+\Lip_b^2 t\right) \left(\tau+\Norm{u(s,y_1)}_p\right)\left(\tau+\Norm{u(s,y_2)}_p\right).
\end{align*}

Now apply the same arguments as those in the proof of Theorem 1.7 of \cite{chen.huang:19:comparison} with $k(t)$
replaced by $k_{8p\Lip_\sigma^2,\Lip_b^2}(t)$ to see that
\begin{align*}
  \Norm{u(t,x)}_p \le \left[\tau + \sqrt{3} \: J_+(t,x)\right] H_{8p\Lip_\sigma^2, \Lip_b^2}\left(t;1\right)^{1/2}.
\end{align*}
In particular, if $f$ satisfies the improved Dalang's condition  \eqref{E:Dalang+} for some $\alpha\in(0,1)$, then by
Corollary \ref{C:Growth}, for all $t>0$ and $x\in\R^d$,
\begin{align*}
  \Norm{u(t,x)}_p \le
  \sqrt{3} \left[\frac{|b(0)|}{L_b}\vee\frac{|\sigma(0)|}{L_{\sigma}} + J_+(t,x)\right] \exp\left(C t\max\left(p^{1/\alpha}\Lip_\sigma^{2/\alpha},\Lip_b\right)\right).
\end{align*}
This proves part (b) of Theorem \ref{T:Mom}.
\end{proof}

\subsection{Uniform moment bounds -- Proof of part (c) of Theorem \ref{T:Mom}} \label{SS:MomBdd}

\begin{proof}[Proof of Theorem \ref{T:Mom} (c)]
  Fix arbitrary $T>0$ and recall that $\alpha\in(0,1)$ as in \eqref{E:Dalang+}. The proof relies on
  the factorization lemma (see, e.g., Section 5.3.1 of \cite{da-prato.zabczyk:14:stochastic}), which
  says that
  \begin{align} \label{E:Factorize}
    u(t,x) = (p_t*u_0)(x) + \Psi(t,x) +\Phi(t,x),
  \end{align}
  where
  \begin{align*}
    \Phi(t,x) = & \frac{\sin(\pi \alpha/2)}{\pi} \int_0^t\int_{\R^d} (t-r)^{-1+\alpha/2}p_{t-r}(x-z) Y(r,z)\ud z \ud r \quad \text{with} \\
    Y(r,z) =    & \int_0^r \int_{\R^d} (r-s)^{-\alpha/2} p_{r-s}(z-y)\sigma(u(s,y))W(\ud s,\ud y)
  \end{align*}
  and
  \begin{align*}
    \Psi(t,x) = & \frac{\sin(\pi \alpha/2)}{\pi} \int_0^t \int_{\R^d} (t-r)^{-1+\alpha/2}p_{t-r}(x-z) B(r,z)\ud z \ud r\quad \text{with} \\
    B(t,x) =    & \int_0^t \int_{\R^d} (r-s)^{-\alpha/2} p_{r-s}(z-y)b\left(u(s,y)\right) \ud s\ud y.
  \end{align*}
  It is clear that
  \begin{align*}
    \sup_{(t,x)\in[0,T]\times\R^d}|(p_t*u_0)(x)|^p \le \Norm{u_0}_{L^\infty(\R^d)}^p.
  \end{align*}

  {\noindent\bf Step 1.~} In this step, we will show that
  \begin{align} \label{E:CasePhi}
    \E\left(\sup_{(t,x)\in [0,T]\times\R^d} |\Phi(t,x)|^p\right)
    & \le C \Norm{u_0}_{L^p(\R^d)}^p L_{\sigma}^p \exp\left(CTp\max\left(L_{b},\: p^{1/\alpha}L_{\sigma}^{2/\alpha}\right)\right).
  \end{align}
  Let $p$ and $q$ be a conjugate pair on positive numbers, i.e., $1/p + 1/q=1$, whose values will be
  determined below. By H\"older's inequality, we see that
  \begin{align*}
    | \Phi(t,x) |
    & \le \frac{\sin(\pi \alpha/2)}{\pi} \int_0^t (t-r)^{-1+\alpha/2}\Norm{p_{t-r}(x-\cdot)}_{L^q(\R^d)} \Norm{Y(r,\cdot)}_{L^p(\R^d)}\: \ud r \\
    & \le C \int_0^t (t-r)^{-1+\alpha/2-(1-1/q)d/2} \Norm{Y(r,\cdot)}_{L^p_\rho(\R^d)}\: \ud r\\
    & \leq C \left(\int_0^t (t-r)^{(-1+\alpha/2)q-(q-1)d/2}\ud r\right)^{1/q} \left(\int_0^t \Norm{Y(r,\cdot)}_{L^p(\R^d)}^p\: \ud r\right)^{1/p},
  \end{align*}
  where we have used the fact that $\Norm{p_{t-r}(x-\cdot)}_{L^q(\R^d)}^q\le C (t-r)^{-d(q-1)/2}$ in
  the second inequality. Hence, since
  \begin{align*}
    \left(-1+\alpha/2\right)q-\left(q-1\right)d/2 > -1 \quad \Longleftrightarrow \quad
    p> (2+d)/\alpha,
  \end{align*}
  we have
  \begin{align*}
    \E\left(\sup_{(t,x)\in [0,T]\times\R^d} |\Phi(t,x)|^p\right)
      & \le C_T \int_0^t \E\left(\Norm{Y(r,\cdot)}_{L^p(\R^d)}^p\right)\: \ud r \\
      & = C_T \int_0^t\ud r\int_{\R^d}\ud z\: \E\left(\left|Y(r,z)\right|^p\right).
  \end{align*}
  Notice that
  \begin{align*}
   \Norm{Y(r,z)}_p^2 \le L_{\sigma}^2
      \int_0^r\ud s \iint_{\R^{2d}}\ud y\ud y'\: (r-s)^{-\alpha} f(y-y')
           p_{r-s}(z-y)  \Norm{u(s,y)}_p \\
    \times p_{r-s}(z-y') \Norm{u(s,y')}_p.
  \end{align*}
  Since $b(0) = \sigma(0)=0$, by \eqref{E:MomRough},
  \begin{align*}
    \Norm{u(s,y)}_p \le C \exp\left( CT\max\left(L_{b},\: p^{1/\alpha}L_{\sigma}^{2/\alpha}\right) \right) J_+(s,y).
  \end{align*}
  Combining the above three bounds shows that
  \begin{gather*}
    \E\left(\sup_{(t,x)\in [0,T]\times\R^d} |\Phi(t,x)|^p\right) \le C \exp\left(
    CTp\max\left(L_{b},\: p^{1/\alpha}L_{\sigma}^{2/\alpha}\right) \right) \int_0^t\ud r\int_{\R^d}
    \ud z\: I^{p/2}(r,z)
  \end{gather*}
  with
  \begin{align*}
    I(r,z):=\int_0^r\ud s \iint_{\R^{2d}}\ud y\ud y'\: (r-s)^{-\alpha} p_{r-s}(z-y) J_+(s,y) f(y-y') p_{r-s}(z-y') J_+(s,y').
  \end{align*}
  By the same arguments as the proof of Theorem 1.8 of \cite{chen.huang:19:comparison} (see, in
  particular, the bound for $I_{1,1}(t,x,x')$ on p. 1006 {\it ibid.}), we see that
  \begin{align*}
    I(r,z) & \le J_+^2(r,z) \int_0^r\int_{\R^d} e^{-\frac{(r-s)s|\xi|^2}{r}}(r-s)^{-\alpha}\hat{f}(\xi)\ud\xi \ud s
             \le C J_+^2(r,z) \int_{\R^d} \frac{\hat{f}(\xi)\ud \xi}{(1+|\xi|^2)^{1-\alpha}}.
  \end{align*}
  By H\"older's inequality, we see that
  \begin{align*}
    \int_0^T\ud r\int_{\R^d} \ud z\: J_+^p(r,z)
    \le \int_0^T\ud r \int_{\R^d} \ud x \: p_{2r}(x-z) \int_{\R^d}\ud z |u_0(z)|^p
    = T \Norm{u_0}_{L^p(\R^d)}^p.
  \end{align*}
  Therefore,
  \begin{align*}
    \E\left(\sup_{(t,x)\in [0,T]\times\R^d} |\Phi(t,x)|^p\right)
    & \le C L_{\sigma}^p e^{CTp\max\left(L_{b},\: p^{1/\alpha}L_{\sigma}^{2/\alpha}\right)} \int_0^T\ud r\int_{\R^d} \ud z\: J_+^p(r,z).
  \end{align*}
  Combining the last two inequalities proves \eqref{E:CasePhi}. \bigskip

  {\noindent\bf Step 2.~} In this step, we will show that
  \begin{align} \label{E:CasePsi}
    \E\left(\sup_{(t,x)\in [0,T]\times\R^d} |\Psi(t,x)|^p\right)
    & \le C \Norm{u_0}_{L^p(\R^d)}^p L_b^p \exp\left(CTp\max\left(L_{b},\: p^{1/\alpha}L_{\sigma}^{2/\alpha}\right)\right).
  \end{align}
  By the same arguments as in Step 1, we see that
  \begin{align*}
    \E\left(\sup_{(t,x)\in [0,T]\times\R^d} |\Phi(t,x)|^p\right)
      & = C_T \int_0^t\ud r\int_{\R^d}\ud z\: \E\left(\left|B(r,z)\right|^p\right).
  \end{align*}
  Notice that
  \begin{align*}
    \Norm{B(r,z)}_p
    & \le L_b \int_0^r \int_{\R^d} (r-s)^{-\alpha/2} p_{r-s}(z-y)\Norm{u(s,y)}_p \ud s\ud y                                                                             \\
    & \le C L_b \exp\left( CT\max\left(L_{b},\: p^{1/\alpha}L_{\sigma}^{2/\alpha}\right) \right) \int_0^r \int_{\R^d} (r-s)^{-\alpha/2} p_{r-s}(z-y)J_+(s,y) \ud s\ud y \\
    & \le C L_b \exp\left( CT\max\left(L_{b},\: p^{1/\alpha}L_{\sigma}^{2/\alpha}\right) \right) J_+\left(r,z\right) \int_0^r (r-s)^{-\alpha/2} \ud s,
  \end{align*}
  from which we deduce \eqref{E:CasePsi}. This proves part (c) of Theorem \ref{T:Mom}.
\end{proof}

\subsection{H\"older regularity -- Proof of Corollary \ref{C:Holder}} \label{SS:Holder}

\begin{proof}[Proof of Corollary \ref{C:Holder}]
  Denote the last two parts of right-hand side of \eqref{E:Mild} by $B(t,x)$ and $I(t,x)$. One can
  use the same arguments as those in the proof of Theorem 1.8 of \cite{chen.huang:19:comparison},
  but with the slightly different moment formula \eqref{E:MomRough}, to show that $I \in C^{\alpha
  /2-, \, \alpha-}\left((0,\infty)\times\R^d\right)$. It remains to show that $B \in C^{\alpha /2-,
  \, \alpha-}\left((0,\infty)\times\R^d\right)$. Now choose and fix arbitrary $n>1$ and $p>2$. For
  any $(t,x)$, $(t',x')\in [\sfrac{1}{n},n]\times\R^d$ with $t'>t$, an application of the Minkowski
  inequality shows that
  \begin{align*}
   & \Norm{B(t,x) - B(t',x')}_p  \le C L_b \left(I_1(t,x,x')+ I_2(t,t',x') + I_3(t,t',x') \right), \quad \text{with}    \\
   & \qquad I_1(t,x,x')   = \int_0^t \int_{\R^d} \left|p_{t-s}(x-y) - p_{t-s}(x'-y)\right|  \Norm{u(s,y)}_p \ud s\ud y, \\
   & \qquad I_2(t,t',x')  = \int_0^t \int_{\R^d} \left|p_{t-s}(x'-y) - p_{t'-s}(x'-y)\right| \Norm{u(s,y)}_p \ud s\ud y, \\
   & \qquad I_3(t,t',x')  = \int_t^{t'} \int_{\R^d} p_{t'-s}(x'-y) \Norm{u(s,y)}_p \ud s\ud y.
  \end{align*}
  By the moment formula \eqref{E:MomRough} and by setting $\mu(\ud z) := |u_0|(\ud z) + \tau \ud z$,
  we see that
  \begin{align*}
    I_1(t,x,x')  & \le C\int_0^t\ud s \int_{\R^d}\ud y\int_{\R^d} \mu(\ud z)\: \left|p_{t-s}(x-y) - p_{t-s}(x'-y)\right|p_{s}(y-z), \\
    I_2(t,t',x') & \le C\int_0^t\ud s \int_{\R^d}\ud y\int_{\R^d} \mu(\ud z)\:\left|p_{t-s}(x'-y) - p_{t'-s}(x'-y)\right|p_{s}(y-z), \\
    I_3(t,t',x') & \le C\int_t^{t'}\ud s \int_{\R^d}\ud y\int_{\R^d} \mu(\ud z)\:p_{t'-s}(x'-y) p_{s}(y-z).
  \end{align*}
  It is clear that $\mu$ is a rough initial condition, i.e., condition \eqref{E:InitData} is
  satisfied for $\mu$. Denote $ J_0(t,x)= (p_t*\mu)(x)$. It is straightforward to see that
  $I_3(t,t',x') \le C(t'-t)J_0\left(t',x'\right)$. As for $I_1$ and $I_2$, for any $\alpha\in
  (0,1)$, by Lemma 3.1 of \cite{chen.huang:19:comparison}, we have that
  \begin{align*}
    I_1(t,x,x')
    & \le C|x-x'|^{\alpha}\int_0^t\frac{\ud s }{(t-s)^{\alpha/2}}\int_{\R^d}\ud y\int_{\R^d} \mu(\ud z)\:\left[p_{2(t-s)}(x-y) + p_{2(t-s)}(x'-y)\right]p_{2s}(y-z), \\
    & = C |x-x'|^{\alpha} t^{1-\alpha/2} \left(J_0(2t,x) + J_0(2t,x')\right),
  \end{align*}
  and similarly,
  \begin{align*}
    I_2(t,t',x')
    & \le C(t'-t)^{\alpha/2}\int_0^t\frac{\ud s }{(t'-s)^{\alpha/2}}\int_{\R^d}\ud y\int_{\R^d} \mu(\ud z)\: p_{4(t'-s)}(x'-y)p_{4s}(y-z), \\
    & \le C(t'-t)^{\alpha/2}J_0\left(4t,x'\right).
  \end{align*}
  Combining the above bounds proves Corollary \ref{C:Holder}.
\end{proof}

\section{Proof of Theorem \ref{T:Main}} \label{S:Main}

\begin{proof}[Proof of Theorem \ref{T:Main}]
  For $N\ge 1$, let us consider the truncated stochastic heat equation:
  \begin{equation}\label{E:trunSHE}
    \begin{aligned}
      &  & u_N(t,x) = \left(p_t*u_0\right)(x) + \int_0^t \int_{\R^d} p_{t-s}(x-y) b_N(u_N(s,y)) \ud y \ud s \\
      &  & + \int_0^t \int_{\R^d} p_{t-s}(x-y)\sigma_N(u_N(s,y))W(\ud s, \ud y)\,,
    \end{aligned}
  \end{equation}
  where
  \begin{align} \label{E:Trun}
    \sigma_N(x) = \sigma\left(\left(1\wedge \frac{N}{|x|}\right) x\right) \quad \text{and} \quad
    b_N(x)      = b\left(\left(1\wedge \frac{N}{|x|}\right) x\right).
  \end{align}
  Recall that $L_{b_N}$ and $L_{\sigma_N}$ denote the growth rate; see \eqref{E:Lb}. According to
  Theorem 1.1 of \cite{huang:17:on}, there exists a unique solution $\{u_N(t,x): t>0,x\in\R^d\}$ to
  \eqref{E:trunSHE}. In the following, we will use $C$ to denote a generic constant that may change
  its value at each appearance, does not depend on $(N,t,x,\epsilon)$, but may depend on $(p,
  \alpha)$. \bigskip

  {\noindent \bf Step 1.~} In this step, we will prove (a). For any $T>0$ fixed, consider the
  following stopping time
  \begin{align*}
    \tau_N := \inf \left\{ t> 0: \sup_{x\in \R^d} |u_N(t,x)|\geq N\right\}\wedge T\,.
  \end{align*}
  Noticing that for all $M\ge N$, we have that $\tau_N \le \tau_M$ and
  \begin{align*}
    u_N(t,x) = u_M(t,x) \quad
    \text{a.s. on $(t,x)\in \left[0,\tau_N\right)\times \R^d$},
  \end{align*}
  we can construct the solution $u(t,x)$ via
  \begin{align} \label{E:uuN}
    u(t,x) = u_N(t,x)\,, \quad
    \text{for all $N\ge 1$ and $(t,x)\in \left[0,\tau_N\right)\times \R^d$}.
  \end{align}
  From the definition, it is clear that on $0 \leq t \leq \tau_N$,
  \begin{align*}
    b_N(u_N(t,x))      = b(u_N(t,x))      = b(u(t,x)) \quad \text{and} \quad
    \sigma_N(u_N(t,x)) = \sigma(u_N(t,x)) = \sigma(u(t,x)).
  \end{align*}
  By the Chebyshev inequality and the moment formula \eqref{E:MomBdd},
  \begin{gather}
    \mathbb{P}\left( 0 \le \tau_N<T \right) = \mathbb{P} \left(\sup_{(t,x) \in \left[0,T\right] \times \R^d} |u_N(t,x)| \geq N\right) \le \frac{1}{N^p} \E \left(\sup_{(t,x) \in \left[0,T\right] \times \R^d} |u_N(t,x)|^p\right) \notag \\
    \le \frac{C}{N^p}\left(\Norm{u_0}_{L^\infty}^p + C \Norm{u_0}_{L^p}^p \left(L_{b_N}+L_{\sigma_N}\right)^p \exp\left(C p T \max\left(L_{b_N},\: p^{1/\alpha}L_{\sigma_N}^{2/\alpha}\right)\right)\right).
    \label{E:TauNT}
  \end{gather}
  The sub-critical conditions in \eqref{E:subCrtcl} implies that
  \begin{equation*}
    L_{b_N} = o\left(\log N\right)  \quad \text{and} \quad
    L_{\sigma_N} = o\left(\left(\log N\right)^{\alpha/2}\right),
  \end{equation*}
  which ensure that above probability in \eqref{E:TauNT} goes to zero as $N\to\infty$. Therefore, by
  sending $N\to\infty$, we see that $u(t,x)$ is well defined on $\left(0,T\right]\times \R^d$. The
  uniqueness is inherited from the uniqueness of $u_N(t,x)$ in \eqref{E:trunSHE}. \bigskip

  {\noindent \bf Step 2.~} Now we prove part (b), the proof of which is similar to that of part (a).
  Fix an arbitrary $T_0>0$. Denote
  \begin{align*}
    \tau_N := \inf \left\{ t> 0: \sup_{x\in \R^d} |u_N(t,x)|\geq N\right\}\wedge T_0\,.
  \end{align*}
  We claim that
  \begin{align} \label{E:tau_1}
    \lim_{N \to \infty} \mathbb{P}\left( 0 \le \tau_N<T\right) =0\,,
    \quad \text{for some non-random constant $T>0$.}
  \end{align}
  Indeed, for all $\epsilon> 0$, by replacing $T$ by $\epsilon$ in \eqref{E:TauNT}, we see that
  \begin{align} \label{E:TauNT-b}
    \mathbb{P}\left( 0 \le \tau_N<\epsilon\right)
    \le & \frac{C}{N^p}\left(\Norm{u_0}_{L^\infty}^p + C \Norm{u_0}_{L^p}^p
      \left(L_{b_N}+L_{\sigma_N}\right)^p
    \exp\left(C p \epsilon \max\left(L_{b_N},\: p^{1/\alpha}L_{\sigma_N}^{2/\alpha}\right)\right)\right).
  \end{align}
  By the critical conditions in \eqref{E:Crtcl}, for some $C>0$,
  \begin{align*}
    L_{b_N} \le C \log N \quad \text{and} \quad
    L_{\sigma_N} \le C (\log N)^{\alpha/2}.
  \end{align*}
  Hence, when $\epsilon$ is small enough, by plugging the above constants into \eqref{E:TauNT-b}, we
  see that the probability in \eqref{E:TauNT-b} goes to zero as $N\to \infty$. Therefore, by
  choosing any positive constant $T \in \left(0,\epsilon\right)$, we prove the claim
  \eqref{E:tau_1}. The uniqueness is proved in the same way as the proof of part (a). \bigskip

  {\noindent \bf Step 3.~} Finally, the H\"older continuity of the solution of $u$ inherits that of
  $u_N$ thanks to their relation given in \eqref{E:uuN}, where the H\"older regularity of $u_N$ with
  given exponents is proved in Corollary \ref{C:Holder}. This completes the proof of Theorem
  \ref{T:Main}.
\end{proof}

\section*{Acknowledgement}
J. Huang thanks Mohammud Foondun for pointing out the reference \cite{salins:21:global} when J. H.
presented this paper at a conference.


\printbibliography[title={References}]

\end{document}
